\documentclass[reqno]{amsart}

\usepackage{amssymb,latexsym,amsmath,amsthm,amsbsy}
\usepackage[mathscr]{eucal}
\usepackage{framed,color,graphicx}
\usepackage{mathrsfs}
\usepackage{tikz}
\usepackage{booktabs}

\usetikzlibrary{positioning}

\tikzset{%
  element/.style={draw, shape=circle, fill=white, inner sep=1.4pt}
}

\DeclareSymbolFont{bbold}{U}{bbold}{m}{n}
\DeclareSymbolFontAlphabet{\mathbbold}{bbold}

\theoremstyle{plain}
\newtheorem{theorem}{Theorem}[section]
\newtheorem{lemma}[theorem]{Lemma}
\newtheorem{corollary}[theorem]{Corollary}
\newtheorem{proposition}[theorem]{Proposition}

\newtheorem{remark}[theorem]{Remark}

\newcommand{\A}{\mathbf A}

\newcommand{\F}{\mathbf F}

\newcommand{\abs}[1]{\lvert#1\rvert}

\renewcommand{\ge}{\geqslant}
\renewcommand{\le}{\leqslant}

\newcommand{\V}{\mathcal{V}}

\newcommand{\id}{\approx}
\newcommand{\Z}{\mathbb Z}
\newcommand{\Sc}{S_{\mathrm c}}

\begin{document}

\title[Finite Bases for Truncated Cyclic Group Flat Semirings]
{Finite Bases for Truncated Cyclic Group Flat Semirings with Two Independent Parameters}

\author{Qirui Ma}
\address{School of Mathematics, Northwest University, Xi'an, 710127, Shaanxi, P.R. China}
\email{qiruimath@yeah.net}

\author{Zidong Gao}
\address{School of Mathematics, Northwest University, Xi'an, 710127, Shaanxi, P.R. China}
\email{zidonggao@yeah.net}


\begin{abstract}
For positive integers \(m,n\), let
\[
A_{m,n}=(\{1,\ldots,m\}\times\Z_n)\cup\{0\}
\]
have flat addition and multiplication truncated at degree \(m\). We prove
that \(A_{m,n}\) is finitely based exactly when \(m\le 2\) or
\((m,n)=(3,1)\). Explicit finite bases are supplied throughout this
region. Outside it, high-girth hypergraphs with a constant-sum rigidity
property yield finite countermodels to every bounded-variable fragment of
the equational theory. The proof places no divisibility or coprimality
restriction on the parameters.
\end{abstract}

\maketitle

\section{Introduction and classification}

All signatures in this paper consist of two binary operations
\((+,\cdot)\), with no constants. The zero occurring in a flat semiring is
absorbing for both operations and is the greatest element of the additive
semilattice. It is not an additive identity.

For independent positive integers \(m,n\), define
\[
A_{m,n}=\bigl(\{1,\ldots,m\}\times\Z_n\bigr)\cup\{0\},
\qquad
\Z_n=\Z/n\Z.
\]
The element \(0\) is newly adjoined. Addition satisfies \(s+s=s\) and
\(s+t=0\) for \(s\ne t\). Multiplication by \(0\) gives \(0\), and
\begin{equation}\label{product}
(i,g)(j,h)=
\begin{cases}
(i+j,g+h), & i+j\le m,\\
0, & i+j>m.
\end{cases}
\end{equation}
Thus \(\abs{A_{m,n}}=mn+1\). The diagonal case \(A_{p,p}\) is the integer
extension of the prime-parameter family in \cite[Section 4]{GRsource}.
The present paper separates the truncation length from the group modulus.

\begin{theorem}\label{main}
For all positive integers \(m,n\),
\[
A_{m,n}\text{ is finitely based}
\quad\Longleftrightarrow\quad
m\le 2\ \text{or}\ (m,n)=(3,1).
\]
Finite bases for all the positive cases are given in
Theorems~\ref{null}, \ref{rowtwo} and~\ref{cube}. In every other case and
for every \(r\ge1\), there exists a finite ai-semiring satisfying all
identities of \(A_{m,n}\) in at most \(r\) variables but not belonging to
\(\V(A_{m,n})\).
\end{theorem}

\begin{center}
\begin{tabular}{lcc}
\toprule
 & \(n=1\) & \(n\ge2\)\\
\midrule
\(m=1,2\) & Finitely based & Finitely based\\
\(m=3\) & Finitely based & Nonfinitely based\\
\(m\ge4\) & Nonfinitely based & Nonfinitely based\\
\bottomrule
\end{tabular}
\end{center}

The general theory of flat algebras developed by Jackson \cite{Jackson08}
and the hypergraph methods of Ham and Jackson \cite{HamJackson} provide
the background. We use the flat-variety theorem and hyperforest
membership theorem of Jackson, Ren and Zhao \cite{JRZ}. Related
structural and limit-variety results appear in \cite{RJL,GRnil}. When
\(n=1\), our classification agrees with the published classification of
finite \(\Sc(W)\) by Wu, Zhao and Ren \cite{Wu}:
\(A_{m,1}\cong\Sc(a^m)\) is finitely based exactly for \(m\le3\). We
nevertheless give direct proofs for that boundary column as part of the
uniform argument.

For \(m\ge3,n\ge2\), a squarefree divisor supplies local membership in
\(\V(A_{m,n})\). The separating identity requires more than the regularity
calculation available on the diagonal. We construct hypergraphs for which
equal edge sums force any assignment with a bounded number of values to
be constant. A probabilistic proof of this property is included in full.
The argument uses elementary concentration and alteration techniques as
in \cite{AlonSpencer}.

\section{Algebraic preliminaries}

An ai-semiring satisfies the following six identities, whose set we denote
by \(\A\):
\begin{align*}
(x+y)+z&\id x+(y+z),&x+y&\id y+x,&x+x&\id x,\\
(xy)z&\id x(yz),&x(y+z)&\id xy+xz,&(x+y)z&\id xz+yz.
\end{align*}
Write \(x\le y\) when \(x+y=y\). A nonempty multiplicative ideal that is
upward closed in this order can be collapsed to a single class to form a
semiring quotient. In a flat semiring every nonempty multiplicative ideal
contains \(0\) and is upward closed.

\begin{lemma}\label{basic}
For all \(m,n\ge1\), the algebra \(A_{m,n}\) is a commutative flat
ai-semiring with \(A_{m,n}^{m+1}=\{0\}\) and \(A_{m,n}^m\ne\{0\}\).
\end{lemma}
\begin{proof}
A nonzero product is obtained by adding the positive first coordinates and
the group coordinates, provided the first-coordinate sum is at most \(m\).
Otherwise it is \(0\). This rule is independent of bracketing and is
commutative. If \(ab=ac\ne0\), cancellation in the integer and group
coordinates gives \(b=c\). If \(b\ne c\), it follows that
\(a(b+c)=0=ab+ac\); the case \(b=c\) follows from idempotence. Right
distributivity is analogous. The nilpotence assertions follow from
positivity and \((1,\overline0)^m=(m,\overline0)\ne0\).
\end{proof}

Let \(\F\) denote the sixteen identities
\begin{equation}\label{flat}
x_1ux_2+y_1uy_2+y_1vy_2\id x_1vx_2+y_1uy_2+y_1vy_2,
\end{equation}
with each of \(x_1,x_2,y_1,y_2\) independently retained or omitted.
Omission introduces no multiplicative identity. By \cite[Lemma 2.1]{JRZ},
\(\A\cup\F\) defines the variety generated by flat semirings, and every
nontrivial subdirectly irreducible member of that variety is flat.
References to numbered results of \cite{JRZ} follow its arXiv version. We
use equational completeness and subdirect representation in the standard
forms of \cite{BS}.

\begin{lemma}\label{annihilator}
A nontrivial subdirectly irreducible commutative flat nilpotent semiring
has exactly one nonzero annihilator element \(\omega\). Every other
nonzero element divides \(\omega\), allowing the element itself as a
divisor.
\end{lemma}
\begin{proof}
Starting from any nonzero element, extend a nonzero product as long as
possible. Nilpotence bounds its length, and a maximal extension is
annihilated by every element. Thus a nonzero annihilator exists. Two
distinct nonzero annihilators \(v,w\) would give nontrivial ideal
congruences collapsing \(\{0,v\}\) and \(\{0,w\}\) whose intersection is
equality, contradicting subdirect irreducibility. The same extension
argument starting from any nonzero element terminates at this unique
\(\omega\).
\end{proof}

\section{Finite bases at truncation lengths one and two}

\subsection{Null multiplication}

\begin{theorem}\label{null}
For every \(n\ge1\), a finite basis for \(A_{1,n}\) is
\[
\A\cup\{xy\id zt,\quad x+yz\id yz\}.
\]
In particular all the semirings \(A_{1,n}\) generate the same variety.
\end{theorem}
\begin{proof}
All products in \(A_{1,n}\) are \(0\), so the identities hold. In an
arbitrary model, \(xy\id zt\) makes multiplication constant, with value
\(o\), and \(x+yz\id yz\) makes \(o\) the additive maximum. Every term is
therefore either \(o\) or a nonempty sum of distinct variables. Different
variable sets are distinguished in the subsemiring
\(\{(1,\overline0),0\}\) of \(A_{1,n}\) by assigning a variable in their
symmetric difference to \(0\) and all other variables to
\((1,\overline0)\). Assigning all variables to \((1,\overline0)\)
separates a variable sum from \(o\). This proves completeness.
\end{proof}

\subsection{Loop and matching semirings}

For cardinals \(s,t\), let \(R_{s,t}\) have elements \(0,\omega\), \(s\)
loop vertices, and \(t\) pairs of vertices. Each loop has square
\(\omega\), and the product of the two vertices of each pair is
\(\omega\). All other products are \(0\), and addition is flat. In
particular, \(R_{0,0}=\{0,\omega\}\). Write
\[
\mathcal B=\A\cup\F\cup\{xy\id yx,\quad x_1x_2x_3\id y_1y_2y_3\}.
\]
This has 24 indexed identities, with no claim of minimality.

\begin{lemma}\label{three-nil}
Every nontrivial subdirectly irreducible model of \(\mathcal B\) is
isomorphic to some \(R_{s,t}\).
\end{lemma}
\begin{proof}
The model is flat and every triple product is \(0\). By
Lemma~\ref{annihilator}, there is a unique nonzero annihilator
\(\omega\). Every vertex outside \(\{0,\omega\}\) has a partner with
product \(\omega\). The partner is unique: in a flat semiring,
\(xy=xz\ne0\) with \(y\ne z\) contradicts \(x(y+z)=xy+xz\).
Commutativity makes the partner map an involution. Its fixed points are
loops and its other orbits are pairs.
\end{proof}

\begin{lemma}\label{amplification}
For finite \(s,t\), the following inclusions hold:
\[
\begin{array}{ll}
R_{s,0}\in\V(R_{2,0}),&R_{0,t}\in\V(R_{0,1}),\\
R_{1,t}\in\V(R_{1,1}),&R_{s,t}\in\V(R_{2,1}).
\end{array}
\]
The corresponding statements hold for infinite index sets whenever the
displayed restriction on the number of loops is respected.
\end{lemma}
\begin{proof}
For \(R_{s,0}\), use a coordinate for every function from the \(s\) loop
indices into the two loop vertices of \(R_{2,0}\). The tuple for an index
records its image. Its square is constantly \(\omega\), while any two
different tuples differ at a coordinate, making both their sum and their
product zero there.

For \(R_{0,t}\), use all binary functions on the \(t\) pair indices. The
two tuples of a pair take opposite endpoints of the unique pair of
\(R_{0,1}\) in every coordinate. Their product is constantly \(\omega\);
squares are zero. Two tuples from different pairs agree somewhere, making
their product zero there, and distinct tuples differ somewhere, making
their sum zero there.

For \(R_{1,t}\) use the same pair construction in \(R_{1,1}\) and add the
constant loop tuple. Cross products between loop and pair tuples are zero.
For \(R_{s,t}\) use all independent choices of one of the two loops for
each loop index and of an orientation for each pair index in \(R_{2,1}\).
The preceding separations hold simultaneously.

In each construction take the generated subsemiring and collapse all
tuples having a zero coordinate. This is an upward-closed multiplicative
ideal. The only remaining tuples are the specified vertices and the
constant \(\omega\) tuple: all longer monomials are zero, all unwanted
quadratic monomials have a zero coordinate, and so do sums of distinct
remaining tuples. The quotient is the claimed semiring. Cases with no
vertex generators reduce to the common subsemiring \(R_{0,0}\). Infinite
cases follow because every finite subset lies in a subsemiring containing
finitely many complete components, and identities are evaluated on
finitely many elements.
\end{proof}

\begin{theorem}\label{rowtwo}
For \(n\ge1\), a finite basis for \(A_{2,n}\) is \(\mathcal B\) together
with the entry in the following table:
\begin{center}
\begin{tabular}{ll}
\toprule
Modulus & Additional identity\\
\midrule
\(n=1\) & \(xy\id x^2+y^2\)\\
\(n=2\) & \(x^2+yz\id x^2+yz+y^2\)\\
\(n\ge3\) odd & \((x+y)^2\id x^2+y^2\)\\
\(n\ge4\) even & none\\
\bottomrule
\end{tabular}
\end{center}
The varieties generated in these four cases are, respectively,
\[
\V(R_{1,0}),\quad \V(R_{2,0},R_{0,1}),\quad
\V(R_{1,1}),\quad \V(R_{2,1}).
\]
\end{theorem}
\begin{proof}
For \(r\in\Z_n\), collapse the ideal consisting of \(0\) and all
degree-two elements except \((2,r)\). The resulting quotient \(T_r\) has
the \(n\) degree-one vertices, with partner involution
\(g\mapsto r-g\), and annihilator \((2,r)\). The kernels of these
quotient maps have equality as their intersection, so
\[
\V(A_{2,n})=\V(T_r:r\in\Z_n).
\]
Loops solve \(2g=r\). If \(n=1\), the quotient is \(R_{1,0}\). If
\(n=2\), the two quotients are \(R_{2,0}\) and \(R_{0,1}\). For odd
\(n\ge3\), each quotient has one loop and \((n-1)/2\) pairs, contains
\(R_{1,1}\), and belongs to its variety by Lemma~\ref{amplification}.
For even \(n\ge4\), the quotients with solvable \(2g=r\) have two loops
and \((n-2)/2\ge1\) pairs; the others have only pairs.
Lemma~\ref{amplification} gives exactly the four stated varieties.

It remains to identify the models of the proposed bases. By
Lemma~\ref{three-nil}, it suffices to examine \(R_{s,t}\). The identity
\(xy\id x^2+y^2\) rules out pairs and permits at most one loop. The
second identity rules out simultaneous loops and pairs: substitute a loop
for \(x\) and a pair for \(y,z\); conversely this is its only possible
failure. The third identity permits at most one loop, since two distinct
loop vertices have squares \(\omega\) but their sum is \(0\); all other
assignments satisfy it. With no additional identity there is no
restriction. These conditions and Lemma~\ref{amplification} show that
every subdirectly irreducible model belongs to the appropriate variety
above. Validity in that variety and subdirect representation prove
completeness.
\end{proof}

\section{An explicit basis for the exceptional semiring}\label{exception-section}

The algebra \(A_{3,1}\) is \(C=\{a,a^2,a^3,0\}\cong\Sc(a^3)\), with
\(a^4=0\). Its finite basability is already a consequence of \cite{Wu}.
Here is a direct basis and completeness proof.

\begin{theorem}\label{cube}
An identity basis for \(A_{3,1}\) is \(\A\cup\F\) together with
\begin{align}
xy&\id yx,\label{cC}\\
x_1x_2x_3x_4&\id y_1y_2y_3y_4,\label{cN}\\
(x+y)^2&\id x^2+y^2,\label{cP2}\\
(x+y)^3&\id x^3+y^3,\label{cP3}\\
x^2+yz&\id x^2+yz+y^2,\label{cE}\\
x^2+y^3&\id x^4,\label{cD}\\
xyz&\id xyz+x^3.\label{cR}
\end{align}
Thus 29 indexed identities suffice in the binary signature.
\end{theorem}
\begin{proof}
\emph{Validity.} The only nonzero square in \(C\) is \(a^2\), attained
only at \(a\), and the only nonzero cube is \(a^3\), also attained only
at \(a\). This proves \eqref{cP2} and \eqref{cP3}. If
\(x^2+yz\ne0\), then \(x^2=yz=a^2\), forcing \(y=z=a\); this proves
\eqref{cE}. A square and a cube can never have a common nonzero value,
proving \eqref{cD}. A nonzero triple product requires \(x=y=z=a\),
proving \eqref{cR}. Commutativity, four-nilpotence and the flat
identities hold as well.

\emph{Subdirectly irreducible models.} Let \(S\) be a nontrivial
subdirectly irreducible model. It is commutative, flat and
four-nilpotent, with a unique nonzero annihilator \(\omega\) by
Lemma~\ref{annihilator}.

Suppose first that \(S^3=\{0\}\). By Lemma~\ref{three-nil},
\(S=R_{s,t}\). Identity \eqref{cP2} permits at most one loop, and
\eqref{cE} prevents loops and pairs from coexisting. Hence \(S\) is
\(R_{1,0}\) or \(R_{0,t}\), including \(t=0\).

Otherwise some \(xyz\ne0\). Four-nilpotence implies \(xyz=\omega\), and
\eqref{cR}, applied with each permutation of the variables, gives
\(x^3=y^3=z^3=\omega\). Identity \eqref{cP3} shows that any two elements
with nonzero cube must be equal: if \(u\ne v\) and
\(u^3=v^3=\omega\), its two sides are \(0\) and \(\omega\). Thus there
is exactly one such element \(a\), every nonzero triple product is
\(a^3=\omega\), and \(a,a^2,\omega,0\) are distinct.

If a nonzero product \(uv\) is not an annihilator, it has a nonzero
extension \(uvw\). The preceding argument forces \(u=v=w=a\), so
\(uv=a^2\). All other nonzero products are \(\omega\). Nonzero
cancellation gives \(av=a^2\) only for \(v=a\) and \(av=\omega\) only
for \(v=a^2\); similarly, \(a^2v\ne0\) only for \(v=a\). Every
remaining vertex has a partner among the remaining vertices, because it
is not an annihilator. Its square cannot be \(\omega\), by \eqref{cD}
with \(y=a\). Thus the remaining vertices form disjoint pairs with
product \(\omega\).

Write \(C_t\) for the semiring formed by adjoining \(t\) such pairs to
the cubic chain \(C\), all other new products being zero. We have proved
that \(S\) is one of \(R_{1,0}\), \(R_{0,t}\) or \(C_t\), with possibly
infinite \(t\).

\emph{Realising the models in \(\V(C)\).} First
\(R_{1,0}\cong C/\{0,a^3\}\). For finite \(t\ge1\), use the coordinate
set \(T=\{1,2\}^{\{1,\ldots,t\}}\) and work in \(C^T\). Let \(z\) be
constantly \(a\), and define
\[
b_i^0(f)=a^{f(i)},\qquad b_i^1(f)=a^{3-f(i)}.
\]
Let \(B\) be the subsemiring generated by \(z\) and these tuples, and put
\(w=z^3\). All coordinates of \(w\) equal \(a^3\). A nonempty monomial
with exponents \(c\) on \(z\) and \(e_i^0,e_i^1\) on the pair tuples can
equal \(w\) only if
\[
c+\sum_i\bigl(e_i^0 f(i)+e_i^1(3-f(i))\bigr)=3
\quad\text{for every }f\in T.
\]
Varying one coordinate choice gives \(e_i^0=e_i^1\) for every \(i\), and
then \(c+3\sum_i e_i^0=3\). Thus the monomial must be either \(z^3\) or
\(b_i^0b_i^1\) for one \(i\).

It follows that the divisors of \(w\) in \(B\), including \(w\) itself,
are exactly
\[
D=\{z,z^2,w\}\cup\{b_i^0,b_i^1:1\le i\le t\}.
\]
To justify the statement for arbitrary elements rather than only
monomials, write any equation \(bc=w\) as a product of sums of monomials.
Each cross product must equal \(w\) coordinatewise. Fixing one factor and
using nonzero cancellation in every coordinate shows that all monomials
in the other factor have the same value. The monomial classification then
gives the displayed list. If an element equals \(w\) without a further
factor, flatness likewise forces all its monomials to equal \(w\).

The complement \(J=B\setminus D\) is a multiplicative ideal, since a
factor of a divisor of \(w\) is again a divisor of \(w\). It is upward
closed: if \(b+c\in D\), the target has no zero coordinate, forcing
\(b=c\) equal to that target. The tuples in \(D\) are distinct, and the
same monomial classification shows that in \(B/J\) the only nonzero
products are the cubic-chain products and \(b_i^0b_i^1=w\). Also any two
distinct tuples in \(D\) have sum in \(J\). Consequently
\(B/J\cong C_t\), so \(C_t\in\V(C)\). The case \(t=0\) is \(C\) itself.
Removing \(a,a^2\) from \(C_t\) gives \(R_{0,t}\) as a subsemiring.
Infinite cases follow by containment of finite subsets in finite
complete-component subsemirings.

All subdirectly irreducible models of the proposed basis therefore lie in
\(\V(C)\). Subdirect representation completes the proof.
\end{proof}

\section{Squarefree divisors}\label{divisor-section}

For \(k\ge1\), let \(Q_k=\Sc(a_1\cdots a_k)\): its nonzero elements are
the nonempty subsets of \([k]\), with disjoint union as the nonzero
multiplication, and its addition is flat. The order is \(2^k\), including
\(0\).

\begin{lemma}\label{protection}
Let \(B\) be a subsemiring of a power of a commutative flat nilpotent
semiring, generated by tuples \(\alpha_1,\ldots,\alpha_k\). Suppose that
every squarefree product
\(\alpha_I=\prod_{i\in I}\alpha_i\), \(\varnothing\ne I\subseteq[k]\),
has no zero coordinate, and that a nonempty monomial evaluating to
\(\alpha_I\) must have exponent vector \(\mathbf1_I\). Then \(Q_k\) is a
quotient of \(B\).
\end{lemma}
\begin{proof}
Let \(D\) consist of these squarefree products and let
\(J=B\setminus D\). The hypotheses imply that the products in \(D\) are
distinct. Nilpotence puts the zero tuple in \(J\). If \(b+c\in D\),
flatness in each coordinate gives \(b=c\) equal to the target, so
\(J+B\subseteq J\).

If \(bc=\alpha_I\), expand both factors into sums of monomials. Every
cross product evaluates to \(\alpha_I\), and hence the sum of its
exponent vectors is \(\mathbf1_I\). Fixing either factor shows that all
monomials in the other factor have the same exponent vector. The two
factors are therefore squarefree products for disjoint nonempty subsets
of \(I\). This proves \(JB\subseteq J\). Thus \(J\) is an upward-closed
ideal. In the quotient, disjoint subsets multiply to their union and
intersecting subsets multiply to \(J\) by uniqueness of exponent vectors.
Distinct products in \(D\) have a zero coordinate in their sum. The
quotient is \(Q_k\).
\end{proof}

\begin{proposition}\label{divisor}
The following inclusions hold:
\begin{enumerate}
\item \(Q_m\in\V(A_{m,n})\) for all \(m\ge1\) and \(n\ge2\);
\item \(Q_{m-1}\in\V(A_{m,1})\) for all \(m\ge2\).
\end{enumerate}
\end{proposition}
\begin{proof}
For (1), in \(A_{m,n}^m\) set
\(\alpha_i(j)=(1,\overline{\delta_{ij}})\). A nonempty monomial of degree
\(d\le m\) has coordinates \((d,\overline e_j)\). Equality with
\(\alpha_I\) forces \(d=\abs I\) and \(e_j\equiv1\pmod n\) for
\(j\in I\), while \(e_j\equiv0\pmod n\) otherwise. Since \(n\ge2\),
each exponent inside \(I\) is at least \(1\). Their total is already
\(\abs I\), so all exponents form \(\mathbf1_I\). All squarefree tuples
have no zero coordinate. Lemma~\ref{protection} applies. Notice that no
inequality between \(m\) and \(n\) was used.

For (2), put \(k=m-1\) and identify \(A_{m,1}\) with \(\Sc(a^m)\). In
its \(k\)th power set \(\alpha_i(j)=a^{1+\delta_{ij}}\). A squarefree
tuple has coordinate exponents
\(\abs I+\mathbf1_I(j)\le k+1=m\). If a degree-\(d\) monomial equals
this tuple, its coordinate exponents give
\[
d+e_j=\abs I+\mathbf1_I(j)\qquad(1\le j\le k).
\]
Summing gives \((k+1)d=(k+1)\abs I\), so \(d=\abs I\) and
\(e=\mathbf1_I\). Again Lemma~\ref{protection} applies.
\end{proof}

\section{Hypergraphs with constant-sum rigidity}\label{rigid-section}

We use Berge girth. A cycle of length \(\ell\ge2\) consists of \(\ell\)
distinct vertices and \(\ell\) distinct hyperedges in alternating
incidence. Girth is the shortest cycle length, or infinity. The incidence
graph has one node for each vertex and each hyperedge and joins them when
incident. A Berge cycle corresponds to an incidence cycle of twice the
length.

\begin{lemma}\label{expansion}
Fix integers \(k\ge3\), \(q\ge2\) and \(L\ge2\). There is a finite
\(k\)-uniform hypergraph \(H=(V,E)\) of girth greater than \(L\), with
no isolated vertices, such that, writing \(N=\abs V\):
\begin{enumerate}
\item every \(U\subseteq V\) with \(\abs U\ge N/q\) contains a
hyperedge;
\item every nonempty \(S\subseteq V\) with
\(\abs S\le(1-1/q)N\) has a hyperedge meeting it in exactly one vertex.
\end{enumerate}
\end{lemma}
\begin{proof}
Let \(\varepsilon=1/(4L)\) and form a random \(k\)-uniform hypergraph
\(H_0\) on \(N\) labelled vertices by including each \(k\)-set
independently with probability
\[
\pi=N^{-(k-1)+\varepsilon}.
\]
All constants below depend only on \(k,q,L\). We show that, with
probability tending to \(1\), three properties hold simultaneously.

\emph{Large subsets have many edges.} For a fixed \(U\) with
\(\abs U\ge N/q\), its number of edges is binomial with mean
\[
\mu_U=\binom{\abs U}{k}\pi\ge c_1N^{1+\varepsilon}
\]
for all sufficiently large \(N\) and a constant \(c_1>0\). The binomial
lower-tail estimate \(\Pr(X<\mu/2)\le\exp(-\mu/8)\) and a union bound
over at most \(2^N\) sets show that, with probability tending to \(1\),
every such \(U\) contains more than \(N\) edges.

\emph{Small subsets have many singleton intersections.} Fix a nonempty
\(S\) of size \(s\le(1-1/q)N\). The number \(X_S\) of edges meeting
\(S\) exactly once is binomial with mean
\[
\mu_S=s\binom{N-s}{k-1}\pi\ge c_2sN^\varepsilon
\]
for a constant \(c_2>0\). For large \(N\), \(\mu_S/2\ge2s\). A union
bound gives
\[
\Pr(\text{some }X_S<2\abs S)
\le\sum_{s=1}^{\lfloor(1-1/q)N\rfloor}
\binom Ns\exp(-c_2sN^\varepsilon/8)=o(1).
\]
Indeed, \(\binom Ns\le(eN/s)^s\le(eN)^s\), and
\(N^\varepsilon/\log N\to\infty\), so the sum is bounded by a geometric
series with ratio tending to zero.

\emph{Short cycles have disjoint edge supports.} Here the support of a
cycle means the union of all vertices in its hyperedges, including
vertices not selected as cycle vertices. Suppose two distinct Berge
cycles of lengths at most \(L\) have intersecting supports. Their union
has at most \(2L\) hyperedges and a connected incidence graph containing
two distinct cycles. If the union has \(e\) hyperedges and \(v\)
vertices, that graph has \(ke\) edges and \(e+v\) nodes. Its cycle rank
is at least \(2\), so
\[
ke-(e+v)+1\ge2,\qquad v\le(k-1)e-1.
\]
There are finitely many possible union types, since \(e\le2L\) and
\(v\le2kL\). The expected number of occurrences of any fixed type is at
most
\[
O(N^v\pi^e)
=O\bigl(N^{v-(k-1)e+\varepsilon e}\bigr)
=O\bigl(N^{-1+2L\varepsilon}\bigr)=O(N^{-1/2}).
\]
This includes cycles of length \(2\) and unions sharing hyperedges.
Cycles differing only by cyclic rotation or reversal are counted as the
same cycle. Markov's inequality proves that, with probability tending to
\(1\), no pair of distinct short cycles has intersecting supports.

Choose a realisation with all three properties. Delete one hyperedge from
each cycle of length at most \(L\). The deleted hyperedges are pairwise
vertex-disjoint; in particular, at most one deleted edge contains any
fixed vertex. Deletion cannot create a cycle, so the remaining hypergraph
\(H\) has girth greater than \(L\). At most \(N\) edges were deleted,
preserving an edge in every large \(U\). At most \(s\) deleted edges
could meet a fixed \(S\) exactly once, so at least \(2s-s=s>0\) such
edges remain. The singleton case gives no isolated vertices. Thus (1)
and (2) hold.
\end{proof}

\begin{corollary}\label{rigidity}
Let \(H\) be as in Lemma~\ref{expansion}. Let \(G\) be any abelian group
and let \(g:V(H)\to G\) have at most \(q\) distinct values. If
\[
\sum_{v\in e}g(v)=b\quad\text{for every }e\in E(H)
\]
with a fixed \(b\in G\), then \(g\) is constant.
\end{corollary}
\begin{proof}
A most frequent value \(a\) occurs on a set \(U\) of size at least
\(N/q\). By property (1), \(U\) contains an edge, giving \(b=ka\). If
\(S=V\setminus U\) is nonempty, property (2) supplies an edge whose
unique vertex \(w\) outside \(U\) has value \(g(w)\ne a\). Its sum is
\((k-1)a+g(w)\), which equals \(ka\) only if \(g(w)=a\), a
contradiction.
\end{proof}

\begin{remark}
The group in Corollary~\ref{rigidity} need not be finite. Only the number
of values used by the assignment is bounded. We will use \(G=\Z_n\) for
nontrivial moduli and \(G=\Z\) for first-coordinate degrees in the
trivial-modulus column.
\end{remark}

\section{Local hyperforest membership}\label{local-section}

For \(k\ge3\) and a finite \(k\)-uniform hypergraph \(H\) of girth at
least \(5\) without isolated vertices, let \(S_H\) be the flat
commutative hypergraph semiring of \cite[Section 3]{JRZ}. Its vertex
generators are \(a_v\). A monomial in them is nonzero precisely when its
vertices are distinct and form a subset of a hyperedge. All full-edge
products have the same value \(\omega\ne0\). The only other
identifications of distinct nonzero monomials are between \((k-1)\)-subsets
linked by a common completing vertex. These are the normal forms of
\cite[Lemma 3.4]{JRZ}; transitivity of linking follows from its
Lemma~3.2. In particular \(a_v^2=0\) and \(S_H\) is finite.

We also use \cite[Lemma 4.2]{JRZ}: for a finite \(k\)-uniform hyperforest
\(F\) without isolated vertices, with \(k>2\), one has
\begin{equation}\label{forest}
S_F\in\V(Q_k).
\end{equation}

\begin{lemma}\label{local}
If \(r\ge1\) and
\[
\operatorname{girth}(H)>\max\left\{5,k\binom{kr}{2}\right\},
\]
then every at most \(r\)-generated subsemiring of \(S_H\) belongs to
\(\V(Q_k)\).
\end{lemma}
\begin{proof}
Let \(T\) be generated by at most \(r\) elements. If all are zero, the
result is immediate. Otherwise choose for each nonzero generator a subset
normal form and a full edge containing it. Let \(X\) be the union of
those edges, so \(\abs X\le kr\). The subsemiring \(T^+\) generated by
\(a_x\) for \(x\in X\) contains \(T\), including any zero generators
since \(a_x^2=0\).

Let \(E_X\) be all edges meeting \(X\) in at least two vertices, set
\(Y=\bigcup E_X\), and put \(F=(Y,E_X)\). Distinct edges of \(H\) meet
in at most one vertex. Choosing a pair from each intersection
\(e\cap X\) therefore injects \(E_X\) into the pairs of \(X\). Hence
\[
\abs{E_X}\le\binom{\abs X}{2},\qquad
\abs Y\le k\binom{kr}{2}.
\]
All originally chosen edges lie in \(E_X\), so \(X\subseteq Y\). A cycle
of \(F\) would be a cycle of \(H\) of length at most \(\abs Y\), a
contradiction. Thus \(F\) is a hyperforest with no isolated vertices.

For monomials supported on \(X\), zero and nonzero values agree in
\(S_H\) and \(S_F\): every edge containing at least two of their vertices
lies in \(E_X\), and repetitions give zero in both. Full-edge products
are identified in both. If two \((k-1)\)-subsets of \(X\) are linked in
\(H\), both completing edges belong to \(E_X\) since \(k-1\ge2\); thus
they are linked in \(F\). Conversely, no new identification arises
because every edge of \(F\) belongs to \(H\). The normal forms show that
equality and multiplication of monomials agree. Flat addition agrees as
well. Therefore the subsemiring of \(S_F\) generated by these vertices is
isomorphic to \(T^+\). Now \eqref{forest} and closure under subalgebras
give \(T\in\V(Q_k)\).
\end{proof}

The enlargement step is the one needed to account for completing vertices
outside \(X\); compare \cite[proof of Theorem 4.9]{JRZ}. It does not
assert that \(S_F\) itself embeds into \(S_H\).

\section{Separating identities and the nonfinite basis region}

For a \(k\)-uniform hypergraph \(H\) and a fixed \(u\in V(H)\), put
\[
t_H=\sum_{e\in E(H)}\prod_{v\in e}x_v,
\qquad
\eta_H:\quad t_H\id t_H+x_u^k.
\]
Fixed orderings and bracketings turn these expressions into terms in the
binary signature. Their choices do not matter in the commutative
semirings considered here.

\begin{lemma}\label{separation}
In either of the following cases, choose \(H\) as in
Lemma~\ref{expansion}, with girth also at least \(5\):
\begin{enumerate}
\item \(m\ge3,n\ge2\), \(k=m\) and \(q=n\);
\item \(m\ge4,n=1\), \(k=m-1\) and \(q=m\).
\end{enumerate}
Then \(\eta_H\) holds in \(A_{m,n}\) and fails in \(S_H\).
\end{lemma}
\begin{proof}
If an evaluation gives \(t_H=0\) in \(A_{m,n}\), both sides of
\(\eta_H\) are zero. Otherwise every edge product has the same nonzero
value. Since there are no isolated vertices, every vertex variable has a
nonzero value \((d_v,g_v)\).

In case (1), each edge has \(m\) vertices and its positive
first-coordinate sum is at most \(m\), so \(d_v=1\) for all vertices.
The second coordinates have a common edge sum in \(\Z_n\). They take at
most \(n\) values; Corollary~\ref{rigidity} therefore gives \(g_v=g\)
for every vertex. Every edge product and \(x_u^m\) then equal
\((m,mg)\), proving \(\eta_H\).

In case (2), the group is trivial. The positive integers
\(d_v\in\{1,\ldots,m\}\) have a common edge sum \(d\le m\). Apply
Corollary~\ref{rigidity} in the additive group \(\Z\); the assignment
uses at most \(m=q\) values. Thus every \(d_v\) equals an integer
\(a\ge1\). Now \(ka=d\le m=k+1\) and \(k\ge3\), so \(a=1\). Each edge
product and \(x_u^k\) have value \((k,\overline0)\), proving the
identity.

In \(S_H\), assign \(x_v=a_v\). Then \(t_H=\omega\ne0\) whereas
\(a_u^k=0\) since \(a_u^2=0\). Thus the right side equals \(0\), and the
identity fails.
\end{proof}

\begin{theorem}\label{nfb}
If \(m\ge3,n\ge2\), or \(m\ge4,n=1\), then \(A_{m,n}\) has no
equational basis with a uniform bound on the number of variables.
\end{theorem}
\begin{proof}
Fix \(r\ge1\) and choose \(k,q\) according to Lemma~\ref{separation}.
Choose \(H\) in Lemma~\ref{expansion} with
\[
L=\max\left\{5,k\binom{kr}{2}\right\}.
\]
By Lemma~\ref{local} and Proposition~\ref{divisor}, every at most
\(r\)-generated subsemiring of \(S_H\) belongs to
\[
\V(Q_k)\subseteq\V(A_{m,n}).
\]
Every evaluation of an identity on at most \(r\) variables lies in such a
subsemiring, so \(S_H\) satisfies all those identities of \(A_{m,n}\).
But Lemma~\ref{separation} gives an identity of \(A_{m,n}\) failing in
\(S_H\). Thus \(S_H\notin\V(A_{m,n})\).

If an identity basis had a uniform variable bound \(r\), the same \(S_H\)
would satisfy that basis and hence belong to the generated variety, a
contradiction. In particular, no finite basis exists.
\end{proof}

\begin{proof}[Proof of Theorem~\ref{main}]
Theorems~\ref{null} and \ref{rowtwo} treat \(m=1,2\),
Theorem~\ref{cube} treats \((m,n)=(3,1)\), and Theorem~\ref{nfb} treats
every remaining pair of positive integers.
\end{proof}

\section{Consequences and scope}

\begin{corollary}
For the diagonal family \(A_p=A_{p,p}\) with \(p\) an arbitrary positive
integer,
\[
A_p\text{ is finitely based}\quad\Longleftrightarrow\quad p\in\{1,2\}.
\]
For fixed \(n\ge2\), the threshold is always \(m=2\), whereas for \(n=1\)
it is \(m=3\).
\end{corollary}

The modulus does influence the equational theory even where it does not
influence finite basability. Theorem~\ref{rowtwo} distinguishes four
varieties in the row \(m=2\). For example, a square-root uniqueness
identity holds for odd moduli and fails for even moduli. This does not
produce an additional finite-basis threshold in that row.

The proof does not infer nonfinite basability merely from the presence of
a nonfinitely based divisor variety. Such an implication is false in
general. The squarefree divisors provide the local membership condition,
while the separately proved identities \(\eta_H\) exclude the entire
hypergraph semirings from the target variety.

Finally, the boundary \((3,1)\) cannot be treated by the nonfinite-basis
construction. There the available squarefree divisor has length
\(m-1=2\), whereas the imported hyperforest theorem and the local linking
argument require \(k\ge3\). Its positive classification has an
independent finite-basis proof in Section~\ref{exception-section}. All
conclusions concern the constant-free binary signature specified at the
outset.

\end{document}